\documentclass[11pt]{amsart}
\usepackage{latexsym,amscd,amssymb,epsfig} 
\usepackage[dvips]{color}
\usepackage[all]{xy}

\theoremstyle{plain}
  \newtheorem{theorem}{Theorem}[section]
  \newtheorem{proposition}[theorem]{Proposition}
  \newtheorem{lemma}[theorem]{Lemma}
  \newtheorem{corollary}[theorem]{Corollary}
  
\theoremstyle{definition}
  \newtheorem{definition}[theorem]{Definition}

\theoremstyle{remark}

\numberwithin{equation}{section}

\def\clap#1{\hbox to 5pt{\hss$#1$\hss}}

\def\umapright#1{\smash{
   \mathop{\longrightarrow}\limits^{#1}}}

\def\rmapdown#1{\Big\downarrow\rlap
   {$\vcenter{\hbox{$\scriptstyle#1$}}$}}
\def\lmapup#1{\llap{$\vcenter{\hbox{$\scriptstyle#1$}}$}
    \Big\uparrow}

\def\tempbaselines
{\baselineskip22pt\lineskip3pt
   \lineskiplimit3pt}
\def\diagram#1{\null\,\vcenter{\tempbaselines
\mathsurround=0pt
    \ialign{\hfil$##$\hfil&&\quad\hfil$##$\hfil\crcr
      \mathstrut\crcr\noalign{\kern-\baselineskip}
  #1\crcr\mathstrut\crcr\noalign{\kern-\baselineskip}}}\,}

\def\pullback#1&#2&#3&#4&#5&#6&#7&#8&{
\diagram{#1&\umapright{#2}&#3\cr
\rmapdown{#4}&&\rmapdown{#5}\cr
#6&\umapright{#7}&#8\cr}}

\def\calC{{\mathcal C}}

\def\calE{{\mathcal E}}

\def\calH{{\mathcal H}}

\def\calP{{\mathcal P}}
\def\calQ{{\mathcal Q}}

\def \End{\mathop{\rm End}\nolimits}

\def \Hom{\mathop{\rm Hom}\nolimits} 
\def \Im{\mathop{\rm Im}\nolimits}

\def \ker{\mathop{\rm Ker}\nolimits} 
\def \op{{\rm op}} 
 
\def\Rad{\mathop{\rm Rad}\nolimits}
\def\Soc{\mathop{\rm Soc}\nolimits}

\def\clap#1{\hbox to 0pt{\hss$#1$\hss}}

\def\ZZ{{\mathbb Z}}

\begin{document}

\title[Perfect complexes]{The Auslander-Reiten Quiver of Perfect Complexes for a Self-Injective Algebra}

\author{Peter Webb}
\email{webb@umn.edu}
\address{School of Mathematics\\
University of Minnesota\\
Minneapolis, MN 55455, USA}

\subjclass[2020]{Primary 16G70; Secondary 18G80, 20C20}

\keywords{Auslander-Reiten triangle, self-injective algebra, perfect complex}

\begin{abstract}
We consider the homotopy category of perfect complexes for a finite dimensional self-injective algebra over a field, identifying many aspects of perfect complexes according to their position in the Auslander-Reiten quiver. Short complexes lie close to the rim. We characterize the position in the quiver of complexes of lengths 1, 2 and 3, as well as rigid complexes and truncated projective resolutions. We describe completely the quiver components that contain projective modules (complexes of length 1). We obtain relationships between the homology of complexes at different places in the quiver, deducing that every self-injective algebra of radical length at least 3 has indecomposable perfect complexes with arbitrarily large homology in any given degree. We show that homology stabilizes, in a certain sense, away from the rim of the quiver. 
\end{abstract}

\maketitle
\section{Introduction}

Let $D^b(\Lambda\hbox{-proj})$ denote the category of perfect complexes for a finite dimensional algebra $\Lambda$ over a field $k$. The objects in this category are the finite-length chain complexes of finitely generated projective $\Lambda$-modules, and we take homotopy classes of chain maps as the morphisms.
It was shown by Wheeler in \cite{Whe} that, if $\Lambda$ is self-injective with no semisimple summands, all components of the Auslander-Reiten quiver of $D^b(\Lambda\hbox{-proj})$ have shape $\ZZ A_\infty$. We exploit this structure to obtain information about the structure of perfect complexes in relation to their position in their quiver component. 

It turns out that the various homology modules of perfect complexes in a quiver component follow a certain pattern. We find that the homology in any fixed degree stabilizes as we move away from the rim of the quiver (Theorem~\ref{homology-picture}). It means that to each perfect complex there is associated a $\Lambda$-module we call the \textit{stabilization module} of the complex, that is the common value to which the zero homology of the complexes in the quiver component stabilizes. The composition factors of the stabilization module are the union of the compositions factors of all the homology modules of a complex on the rim of its component, twisted by a power of the Nakayama functor.

We deduce that, when $\Lambda$ is a self-injective algebra of radical length at least 3, there are always indecomposable complexes with degree-zero homology of arbitrarily large dimension (Theorem~\ref{big-homology-complexes}). We take this statement to be a quantification of the fact that there are very many indecomposable complexes for such an algebra, even including algebras such as $k[t]/(t^3)$. We note that this conclusion does not mention the Auslander-Reiten quiver or even triangulated categories.

In more detail, we completely describe the components of the Auslander-Reiten quiver that contain a projective module (Theorem~\ref{projective-homology-picture}). For arbitrary components we give information based on the structure of complexes close to the rim of the component. In general, a complex of length $n$ must lie at distance at most $n-1$ from the rim, and the only complexes at the maximal possible distance are completely described and lie in a quiver component containing a projective (Corollary~\ref{distance-n-complexes}).
We describe completely the structure of complexes of length 2 or 3 in terms of their distance from the rim (Corollary~\ref{2-term-on-the-rim} and Proposition~\ref{3-term-complexes-position}).

%\babble{Mention that the AR sequence is the degree zero homology of an AR triangle.}
We find that for each Auslander-Reiten sequence of $\Lambda$ modules there is a three-term perfect complex that has the three modules in the sequence as its homology groups (Corollary~\ref{AR-as-homology}). Given three modules appearing in a short exact sequence it is not, in general, possible to realize them as homology of a perfect complex of length 3 (by \cite{ABIM}, for example), so the fact that it can be done for Auslander-Reiten sequences has some interest. 

In proving this result we describe a fundamental relationship between Auslander-Reiten sequences and Auslander-Reiten triangles of perfect complexes that holds for finite dimensional algebras in general (Proposition~\ref{AR-complex}), which is that the degree zero homology of the Auslander-Reiten triangle terminating at a 2-step projective presentation of a non-projective module is the Auslander-Reiten sequence terminating at that module.

In a final section we show that rigid complexes lie on the rim of their quiver component (Theorem~\ref{rigid-complex-theorem}).

With the exception of the application to cohomological Mackey functors in Theorem~\ref{cohomological-MF-theorem}, all the results in this paper appeared as the second half of the preprint \cite{Webd}, which has been withdrawn. The first half of \cite{Webd} has appeared in \cite{Webb}.

%\babble{Cut this paragraph?}
%Perfect complexes arise naturally in many situations.  They are the compact objects in the bounded derived category $D^b(\Lambda\hbox{-mod})$ \cite[6.3]{Ric}. 
%We mention some particular examples of perfect complexes that motivate us. These include the tilting complexes in Rickard's theory \cite{Ric} of derived equivalences of rings. In the context of free group actions on topological spaces, a CW complex on which a group acts freely has a chain complex which is a perfect complex of modules for the group algebra (see \cite{AS} for instance). In the theory of subgroup complexes of a finite group, the chain complex of the poset of non-identity $p$-subgroups, over a $p$-local ring, is homotopy equivalent to a perfect complex of modules for the group algebra \cite{Weba}. These are just a few examples.

\section{Preliminaries}
\label{preliminaries-section}

We take $\Lambda$ to be a finite dimensional algebra over a field $k$, and much of the time (but not always) $\Lambda$ will be self-injective.
Auslander-Reiten triangles in triangulated categories were introduced by Happel in the 1980s, and his book~\cite{Hap} is a good place to read about them. Happel shows that in the bounded derived category $D^b(\Lambda\hbox{-mod})$ an Auslander-Reiten triangle $X\to Y\to Z\to X[1]$ exists if and only if $Z$ is isomorphic to an indecomposable perfect complex, which happens if and only if $X$ is isomorphic to an indecomposable finite complex of finitely generated injective modules. From this it follows that if $\Lambda$ is self-injective then Auslander-Reiten triangles exist in $D^b(\Lambda\hbox{-proj})$ for arbitrary indecomposable perfect complexes $X$ and for arbitrary indecomposable perfect complexes $Z$. Happel describes the construction of these triangles, and it is expressed in terms of the left derived functor of the Nakayama functor $\nu:\Lambda\hbox{-mod}\to\Lambda\hbox{-mod}$ described in \cite{ASS}. The left derived functor is calculated on a perfect complex by applying $\nu$ to each term and each morphism. The resulting functor $D^b(\Lambda\hbox{-proj})\to D^b(\Lambda\hbox{-inj})$ is also denoted $\nu$.

The first lemma is crucial in describing the homology patterns in a quiver component. It follows from \cite[Lemma 2.1]{Webb} and \cite[Lemma 2.2]{DPW}, where more general assertions appear. In the context of stable module categories of self-injective algebras very closely related results have been proved in \cite[Lemma 3.2]{ES} and \cite[Lemma 1.4]{EK}. We give the short proof.

\begin{lemma}\label{split-long-exact-sequence-lemma}
Let $\Lambda$ be a finite dimensional algebra and let
$$
L\to M\to N\to L[1]
$$
be an Auslander-Reiten triangle in $D^b(\Lambda)$. Assume that $N$ is not the shift of a projective module (a complex with only one nonzero term). Then the associated long exact homology sequence is the splice of short exact sequences
$$0\to H_i(L)\to H_i(M)\to H_i(N)\to 0$$
with zero connecting homomorphisms.
\end{lemma}

\begin{proof}
The homology $H_i$ is a representable functor:
$$
H_i(N)=\Hom_{D^b(\Lambda)}(\Lambda[i],N).
$$
The connecting homomorphisms in the long exact sequence of homology arising from the triangle are all zero, except when $N$ is the shift a projective module, because of the lifting property of the Auslander-Reiten triangle, since $N$ is never a summand of $\Lambda[i]$.
\end{proof}

\begin{definition}
We will call a perfect complex $Q_\bullet$ of $\Lambda$-modules \textit{minimal} if, for all complexes $P_\bullet$ homotopic to $Q_\bullet$, we have $\dim Q_n\le \dim P_n$.
\end{definition}

Indecomposable perfect complexes are minimal, if they are not homotopic to $0$. This follows from the next result, in which we recall some well known basic facts about minimal perfect complexes. They are proved (in the dual context of complexes of injectives) in \cite[B2]{Kra} except for the final statement, which is an easy exercise.

\begin{proposition}
Let $\Lambda$ be a finite dimensional algebra over $k$, and $P_\bullet$ a perfect complex of $\Lambda$-modules. Then $P_\bullet$ is homotopy equivalent to a minimal perfect complex, which is unique up to isomorphism. The minimal complex is a direct summand of every complex in its homotopy class. When $\Lambda$ is self-injective, the minimal complex  has the property that its homology is non-zero in the highest and lowest degrees in which the complex has non-zero terms.
\end{proposition}

\begin{definition}

By the \textit{length} of a perfect complex $P_\bullet$ of $\Lambda$-modules we will mean the length of the minimal perfect complex $Q_\bullet$ homotopic to $P_\bullet$. Explicitly, if 
$$
Q_\bullet = \cdots \to 0\to Q_m\to\cdots\to Q_n\to 0\to\cdots
$$
then the length of $P_\bullet$ is $m-n+1$.
If we also assume that $P_\bullet$ (and $Q_\bullet$) are indecomposable, this is the number of non-zero terms in $Q_\bullet$, and if $\Lambda$ is self-injective it is the difference between the degrees of the highest and lowest non-zero homology of $P_\bullet$, plus 1.
\end{definition}

To establish our notation for complexes in an Auslander-Reiten quiver component we record Wheeler's result on its structure for perfect complexes over a self-injective algebra. It was also proved later in \cite{HKR}.

\begin{theorem}[Wheeler, \cite{Whe}]
\label{a-infinity-component-theorem}
Let $\Lambda$ be a self-injective algebra with no semisimple summands. Every component of the Auslander-Reiten quiver of $D^b(\Lambda\hbox{-proj})$ has the form $\ZZ A_\infty$. Specifically, for various perfect complexes $C_0,C_1,\ldots$, it is:
\begin{figure}[h]
%\hfil
{
\def\tempbaselines
{\baselineskip16pt\lineskip3pt
   \lineskiplimit3pt}
\def\diagram#1{\null\,\vcenter{\tempbaselines
\mathsurround=0pt
    \ialign{\hfil$##$\hfil&&\quad\hfil$##$\hfil\crcr
      \mathstrut\crcr\noalign{\kern-\baselineskip}
  #1\crcr\mathstrut\crcr\noalign{\kern-\baselineskip}}}\,}

\def\clap#1{\hbox to 0pt{\hss$#1$\hss}}

$$\diagram{&\clap{\nu C_0[-1]}&&&&\clap{C_0}&&&&\clap{\nu^{-1}C_0[1]}&&&&\clap{\nu^{-2}C_0[2]}\cr
&&\searrow&&\nearrow&&\searrow&&\nearrow&&\searrow&&\nearrow&\cr
\cdots&&&\clap{\nu C_1[-1]}&&&&\clap{C_1}&&&&\clap{\nu^{-1}C_1[1]}&&&\cdots\cr
&&\nearrow&&\searrow&&\nearrow&&\searrow&&\nearrow&&\searrow&\cr
&\clap{\nu^2C_2[-2]}&&&&\clap{\nu C_2[-1]}&&&&\clap{C_2}&&&&\clap{\nu^{-1}C_2[1]}\cr
&&\searrow&&\nearrow&&\searrow&&\nearrow&&\searrow&&\nearrow&\cr
&&&\clap{\vdots}&&&&\clap{\vdots}&&&&\clap{\vdots}\cr
}
$$
}
%\hfil
\caption{Auslander-Reiten quiver component of perfect complexes for a self-injective algebra}\label{AR-quiver-self-injective}
\end{figure}
\end{theorem}

\begin{definition}
The complexes $\nu^{-n}C_0[n]$, $n\in\ZZ$, constitute the \textit{rim} of the quiver component. The complexes $\nu^{-n}C_d[n]$ are said to be at \textit{distance $d$ from the rim}. Given a complex $\nu^{-n}C_d[n]$ its \textit{wing} consists of the complexes $\nu^{-i}C_j[i]$ with $0\le j\le d$ and $-n - d + j\le i \le -n$. These are the complexes lying within a triangle with $\nu^{-n}C_d[n]$ as a vertex and a segment of the rim as the opposite edge. The \textit{wing rim} associated to $\nu^{-n}C_d[n]$ is the part of the rim that lies in the wing. A \textit{mesh} of this quiver diagram is a set of four vertices forming a diamond, whose corresponding complexes form the first three terms of an Auslander-Reiten triangle (in the case of an Auslander-Reiten triangle with first and third terms on the rim, the mesh again consists of the first three terms of that triangle, but there are only three complexes and they are the vertices of a triangular shape.).
\end{definition}

\section{The homology diagram of a quiver component}
\label{quiver-position}

Throughout this section $\Lambda$ will be a finite dimensional self-injective algebra over $k$ with no semisimple direct summand. Thus by Theorem~\ref{a-infinity-component-theorem} all components of the Auslander-Reiten quiver of $D^b(\Lambda\hbox{-proj})$ have the form $\ZZ A_\infty$. We will examine the degree-zero homology of the complexes in each component, showing that it stabilizes sufficiently far from the rim to a single module. We will describe completely the quiver components that contain projective modules (regarded as complexes in degree 0).

We start with some definitions.

\begin{definition}
Let $\Lambda$ be a self-injective algebra. For each $\Lambda$-module $M$ we let $\calP_M$ be the 2-term complex $P_1\to P_0$ in degrees 1 and 0 such that $P_1\to P_0\to M$ is the start of a minimal projective resolution of $M$. For each  simple module $S$ with projective cover $P_S$ note that $\Soc P_S$ is a simple module, and $\nu P_S$ is the injective hull of $S$, which is again an indecomposable projective module. We define $\calH_S$ to be the 3-term complex $\nu^{-1}P_S\to P_S\to\nu P_S$ in degrees 1, 0 and -1, where both maps send the top of one indecomposable projective isomorphically to the simple socle of the next.
We write $H_S=\Rad P_S / \Soc P_S$ for the \textit{heart} of $P_S$.
\end{definition}

\begin{lemma}
Let $\Lambda$ be a self-injective algebra, $M$ a $\Lambda$-module and $S$ a simple $\Lambda$-module.
\begin{enumerate}
\item $H_0(\calP_M)\cong M$ and $H_1(\calP_M)\cong \Omega^2 M$.
\item $\calP_{\Omega^{-1}S}$ is the complex $P_S\to\nu P_S$, where the map sends the simple top of $P_S$ isomorphically to the socle of $\nu P_S$. Similarly $\nu^{-1}\calP_{\Omega^{-1}S}$ is the complex $\nu^{-1}P_S\to P_S$.
\item $H_i(\calH_S) = 
\begin{cases}\Omega^{-1} S& i=-1,\\ 
H_S:=\Rad P_S / \Soc P_S &i=0,\\ 
\Omega \Soc P_S&i=1.\\
\end{cases}
$
\end{enumerate}
\end{lemma}

\begin{proof}
These calculations are immediate from the definitions. 
\end{proof}

The complexes just described appear in the Auslander-Reiten quiver component that contains $P_S$, as the next result shows. The full structure of this quiver component will be described in Theorem~\ref{projective-homology-picture}.

\begin{proposition}
\label{projectives-on-rim}
Let $S$ be a simple non-projective module for a self-injective algebra $\Lambda$ and let $P_S$  be the projective cover of $S$. The Auslander-Reiten quiver of perfect complexes has the following shape close to $P_S$: 
{
\def\tempbaselines
{\baselineskip16pt\lineskip3pt
   \lineskiplimit3pt}
\def
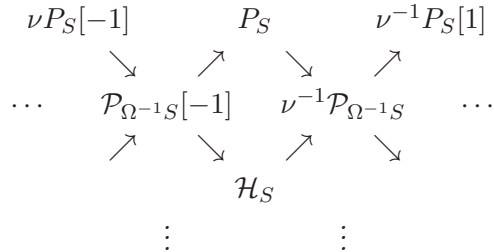
\diagram#1{\null\,\vcenter{\tempbaselines
\mathsurround=0pt
    \ialign{\hfil$##$\hfil&&\quad\hfil$##$\hfil\crcr
      \mathstrut\crcr\noalign{\kern-\baselineskip}
  #1\crcr\mathstrut\crcr\noalign{\kern-\baselineskip}}}\,}
\def\clap#1{\hbox to 0pt{\hss$#1$\hss}}
$$\diagram{&\clap{\nu P_S[-1]}&&&&\clap{P_S}&&&&\clap{\nu^{-1}P_S[1]}\cr
&&\searrow&&\nearrow&&\searrow&&\nearrow&\cr
\cdots&&&\clap{\calP_{\Omega^{-1}S}[-1]}&&&&\clap{\nu^{-1}\calP_{\Omega^{-1}S}}&&&\cdots\cr
&&\nearrow&&\searrow&&\nearrow&&\searrow&\cr
&&&&&\clap{\calH_S}&&&&\cr
%&&\searrow&&\nearrow&&\searrow&&\nearrow&\cr
&&&\clap{\vdots}&&&&\clap{\vdots}\cr
}
$$
}
\end{proposition}

\begin{proof}
The construction of the Auslander-Reiten triangles comes from the description in \cite{Hap}. Thus we have an Auslander-Reiten triangle $\nu P_S[-1]\to \calP_{\Omega^{-1}S}[-1]\to P_S\to\nu P_S$ where the final map sends the top of $P_S$ to the socle of $\nu P_S$. This is because this final map of 1-term complexes is not homotopic to zero and lies in the socle of $\Hom(P_S,\nu P_S)$, and $\calP_{\Omega^{-1}S}$ is the mapping cone of this map. The complex $\calP_{\Omega^{-1}S}$ is indecomposable (or zero in case $P_S = S$), because if it were not it would have to be the direct sum of two one-term complexes $P_S$ and $\nu P_S$ and would have projective homology. However, the homology of $\calP_{\Omega^{-1}S}$ is $\Omega(S)$ and $\Omega^{-1}(S)$ in degrees 1 and 0, and is not projective.

We construct the middle term of the Auslander-Reiten triangle starting at $\calP_{\Omega^{-1}S}[-1]$ as the mapping cone of a map $\nu^{-1}\calP_{\Omega^{-1}S}[-1]\to \calP_{\Omega^{-1}S}[-1]$ which is not homotopic to zero and which lies in the socle of such maps under the action of the endomorphism ring of either complex. Such a map is
$$
\begin{pmatrix}
P_S\cr
\llap{$\scriptstyle\alpha$}\uparrow
\vphantom{\lmapup{\alpha}}\cr 
\nu^{-1}P_S\cr
\end{pmatrix}
\begin{matrix}
\umapright{\nu(\alpha)}\cr
\phantom{\lmapup{\alpha}}\cr
\umapright{0}\cr
\end{matrix}
\begin{pmatrix}
\nu P_S\cr
\llap{$\scriptstyle\nu(\alpha)$}\uparrow
\vphantom{\lmapup{\alpha}}\cr 
\hskip13pt P_S\hskip13pt\cr
\end{pmatrix},
$$
as is readily verified,
and its mapping cone is a complex
$$
\nu^{-1}P_S\xrightarrow{\begin{pmatrix}\scriptstyle\alpha\cr \scriptstyle0\cr\end{pmatrix}}
P_S\oplus P_S \xrightarrow{(\nu(\alpha),\nu(\alpha))}
\nu P_S.
$$
The middle term has a direct summand $\{(x,-x)\bigm| x\in P_S\}\cong P_S$ which is a direct complement to the subcomplex whose middle term is the first $P_S$ direct summand. From this we see that the mapping cone is the direct sum as complexes $P_S\oplus\calH_S$. This completes the description of this part of the Auslander-Reiten quiver.
\end{proof}
\
\begin{corollary}
Let $\Lambda$ be a self-injective algebra. Each component of the Auslander-Reiten quiver of $D^b(\Lambda\hbox{-proj})$ contains at most one indecomposable projective module regarded as a complex in degree 0.
\end{corollary}

\begin{proof}
If the component consists of complexes for a semisimple summand of $\Lambda$ it is the set of shifts of a single module, and so there is only one complex that is a module in degree 0 in the component. Otherwise,  Wheeler \cite{Whe} proved that the component has shape $\ZZ A_\infty$. If it contains a projective module, that module lies on the rim by Proposition~\ref{projectives-on-rim}, and the other terms on the rim are its shifts with the Nakayama functor applied, so no other projectives are possible.
\end{proof}

We will obtain a lot of information from the homology of the complexes in a quiver component.

\begin{definition}
The \textit{homology diagram} of an Auslander-Reiten quiver component is the diagram obtained by replacing each complex in the quiver component by its zero homology module, and each irreducible morphism by the $\Lambda$-module homomorphism it induces on zero homology. We illustrate a homology diagram in Figure~\ref{homology-diagram}.
\end{definition}

\begin{proposition}
We continue the notation of Figure~\ref{AR-quiver-self-injective}.
Let $\Lambda$ be a self-injective algebra.
\begin{enumerate}
\item Identifying terms via the isomorphism
$$
H_0(\nu^jC_i[-j])\cong \nu^jH_j(C_i),
$$
 the homology diagram of the quiver component in Figure~\ref{AR-quiver-self-injective} is as follows:
\begin{figure}[h]
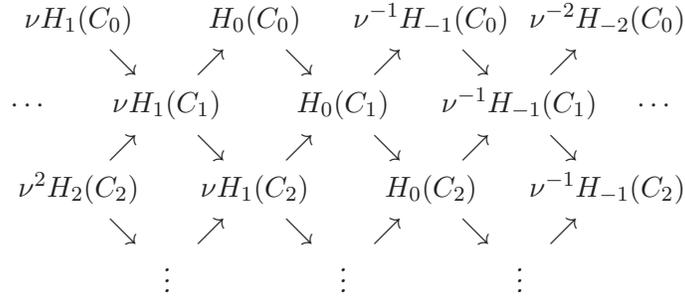

%\hfil
{
\def\tempbaselines
{\baselineskip16pt\lineskip3pt
   \lineskiplimit3pt}
\def\diagram#1{\null\,\vcenter{\tempbaselines
\mathsurround=0pt
    \ialign{\hfil$##$\hfil&&\quad\hfil$##$\hfil\crcr
      \mathstrut\crcr\noalign{\kern-\baselineskip}
  #1\crcr\mathstrut\crcr\noalign{\kern-\baselineskip}}}\,}

\def\clap#1{\hbox to 0pt{\hss$#1$\hss}}

$$\diagram{&\clap{\nu H_1(C_0)}&&&&\clap{H_0(C_0)}&&&&\clap{\nu^{-1}H_{-1}(C_0)}&&&&\clap{\nu^{-2}H_{-2}(C_0)}\cr
&&\searrow&&\nearrow&&\searrow&&\nearrow&&\searrow&&\nearrow&\cr
\cdots&&&\clap{\nu H_1(C_1)}&&&&\clap{H_0(C_1)}&&&&\clap{\nu^{-1}H_{-1}(C_1)}&&&\cdots\cr
&&\nearrow&&\searrow&&\nearrow&&\searrow&&\nearrow&&\searrow&\cr
&\clap{\nu^2H_2(C_2)}&&&&\clap{\nu H_1(C_2)}&&&&\clap{H_0(C_2)}&&&&\clap{\nu^{-1}H_{-1}(C_2)}\cr
&&\searrow&&\nearrow&&\searrow&&\nearrow&&\searrow&&\nearrow&\cr
&&&\clap{\vdots}&&&&\clap{\vdots}&&&&\clap{\vdots}\cr
}
$$
}
%\hfil
\caption{Homology diagram of an Auslander-Reiten quiver component of perfect complexes for a self-injective algebra}\label{homology-diagram}
\end{figure}

\item Every homology module of every complex $C_0, C_1,\ldots$ appears in the homology diagram, up to a twist by the Nakayama functor.
\end{enumerate}
\end{proposition}

\begin{proof}
The Nakayama functor is exact, because it is the composite of the exact functors $\Hom_\Lambda(-,\Lambda)$ (which is exact because $\Lambda$ is injective), and the ordinary duality $\Hom_k(-,k)$. We combine this with the fact that $H_0$ of a shifted complex is homology in a different degree and obtain
$
H_0(\nu^jC_i[-j])\cong \nu^jH_0(C_i[-j]) \cong \nu^jH_j(C_i)
$. The rest is apparent.
\end{proof}

In the next results we will examine the pattern of modules in the homology diagram. The situation is slightly different for quiver components containing a projective module compared to those that do not but the basic idea is the same in both cases. It is a little more complicated in the case of components that contain a projective, so we start with quiver components that do not contain a projective module.

\begin{proposition}\label{homology-pattern}
Let $\Lambda$ be a self-injective algebra and consider the homology diagram of a  quiver component of $D^b(\Lambda\hbox{-proj})$ that does not contain a projective module. 
\begin{enumerate}
\item Every mesh in the diagram consists of modules that lie in a short exact sequence of $\Lambda$-modules. 
\item The composition factors (taken with multiplicity) of $H_0$ of any complex are the union of the composition factors of the $H_0$ of complexes lying on the associated wing rim. 
\end{enumerate}
\end{proposition}

In part (1), what we mean is that, for example, the modules in the mesh
{
\def\tempbaselines
{\baselineskip16pt\lineskip3pt
   \lineskiplimit3pt}
\def\diagram#1{\null\,\vcenter{\tempbaselines
\mathsurround=0pt
    \ialign{\hfil$##$\hfil&&\quad\hfil$##$\hfil\crcr
      \mathstrut\crcr\noalign{\kern-\baselineskip}
  #1\crcr\mathstrut\crcr\noalign{\kern-\baselineskip}}}\,}

\def\clap#1{\hbox to 0pt{\hss$#1$\hss}}

$$\diagram{&&\clap{\nu^{i-1}H_{i-1}(C_{n-2})}&&\cr
&\nearrow&&\searrow&\cr
\clap{\nu^i H_i(C_{n-1})}&&&&{\nu^{i-1}H_{i-1}(C_{n-1})}\cr
&\searrow&&\nearrow&\cr
&&\clap{\nu^i H_i(C_n)}&&\cr
}
$$
}
form a short exact sequence of $\Lambda$-modules
$$
0\to \nu^i H_i(C_{n-1})\to \nu^{i-1}H_{i-1}(C_{n-2}) \oplus\nu^i H_i(C_n)\to \nu^{i-1}H_{i-1}(C_{n-1})\to 0
$$
and similarly with triangular meshes, in which case there is only a single summand in the middle.

\begin{proof}
Because there is no projective module in the quiver component, part (1) is an immediate application of  Lemma~\ref{split-long-exact-sequence-lemma}. 

(2) We consider a module $\nu^i H_i(C_n)$ at distance $n$ from the rim and proceed by induction on $n$. On the rim when $n=0$ the result holds.  In case $n=1$ the exact sequence of homology modules of part (1) implies that the composition factors of the middle term are the union (with multiplicities) of the composition factors of the left and right terms, which is what we have to prove. When $n\ge 2$ the composition factors of $\nu^i H_i(C_n)$ are those of the left and right terms in the mesh above it, with the composition factors of the top term in the mesh removed. By induction these are the composition factors from two parts of the rim (factors common to both parts taken twice) with a copy of the factors from the intersection of those parts of the rim removed. This gives exactly the desired result.
\end{proof}

\begin{theorem}\label{homology-picture}
Let $\Lambda$ be a self-injective algebra and consider the homology diagram of a quiver component of $D^b(\Lambda\hbox{-proj})$ that does not contain a projective module. 
\begin{enumerate}
\item At any two positions in the quiver where the wing rims of the homology diagram have the same non-zero terms, the zero homology modules at those positions are isomorphic. Morphisms in the homology diagram between adjacent such positions are isomorphisms.
\item At any horizontal coordinate, sufficiently far from the rim the homology modules stabilize to a module $\Sigma$ whose composition factors are the union of the composition factors of all terms on the rim.
\item Adjacent to zeros on the rim of the homology diagram, this diagram has the shape shown below, with the maps between identically labelled terms being isomorphisms.  
{
\def\tempbaselines
{\baselineskip11pt\lineskip2pt
   \lineskiplimit2pt}
\def
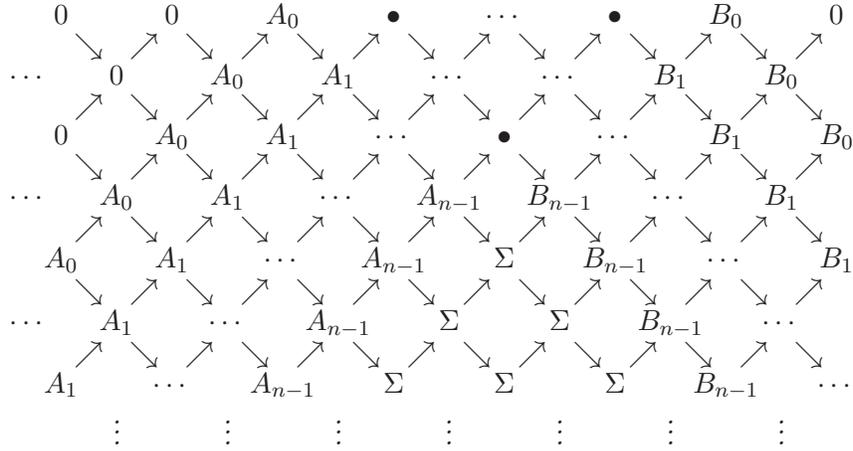
\diagram#1{\null\,\vcenter{\tempbaselines
\mathsurround=0pt
    \ialign{\hfil$##$\hfil&&\hskip5pt\hfil$##$\hfil\crcr
      \mathstrut\crcr\noalign{\kern-\baselineskip}
  #1\crcr\mathstrut\crcr\noalign{\kern-\baselineskip}}}\,}
\def\clap#1{\hbox to 0pt{\hss$#1$\hss}}
$$\diagram{&\clap{0}&&&&\clap{0}&&&&\clap{A_0}&&&&\clap{\bullet}
&&&&\clap{\cdots}&&&&\clap{\bullet}&&&&\clap{B_0}&&&&\clap{0}\cr
&&\searrow&&\nearrow&&\searrow&&\nearrow&&\searrow&&\nearrow&
&\searrow&&\nearrow&&\searrow&&\nearrow&&\searrow&&\nearrow&
&\searrow&&\nearrow&\cr
\cdots&&&\clap{0}&&&&\clap{A_0}&&&&\clap{A_1}&&&&\clap{\cdots}
&&&&\clap{\cdots}&&&&\clap{B_1}&&&&\clap{B_0}\cr
&&\nearrow&&\searrow&&\nearrow&&\searrow&&\nearrow&&\searrow&
&\nearrow&&\searrow&&\nearrow&&\searrow&&\nearrow&&\searrow&
&\nearrow&&\searrow&\cr
&\clap{0}&&&&\clap{A_0}&&&&\clap{A_1}&&&&\clap{\cdots}
&&&&\clap{\bullet}&&&&\clap{\cdots}&&&&\clap{B_1}&&&&\clap{B_0}\cr
&&\searrow&&\nearrow&&\searrow&&\nearrow&&\searrow&&\nearrow&
&\searrow&&\nearrow&&\searrow&&\nearrow&&\searrow&&\nearrow&
&\searrow&&\nearrow&\cr
\cdots&&&\clap{A_0}&&&&\clap{A_1}&&&&\clap{\cdots}&&&&\clap{A_{n-1}}&&&&\clap{B_{n-1}}&&&&\clap{\cdots}
&&&&\clap{B_1}\cr
&&\nearrow&&\searrow&&\nearrow&&\searrow&&\nearrow&&\searrow&
&\nearrow&&\searrow&&\nearrow&&\searrow&&\nearrow&&\searrow&
&\nearrow&&\searrow&\cr
&\clap{A_0}&&&&\clap{A_1}&&&&\clap{\cdots}&&&&\clap{A_{n-1}}&&&&\clap{\Sigma}&&&&\clap{B_{n-1}}
&&&&\clap{\cdots}&&&&\clap{B_1}\cr
&&\searrow&&\nearrow&&\searrow&&\nearrow&&\searrow&&\nearrow&
&\searrow&&\nearrow&&\searrow&&\nearrow&&\searrow&&\nearrow&
&\searrow&&\nearrow&\cr
\cdots&&&\clap{A_1}&&&&\clap{\cdots}&&&&\clap{A_{n-1}}&&&&\clap{\Sigma}&&&&\clap{\Sigma}&&&&\clap{B_{n-1}}&&&&\clap{\cdots}
\cr
&&\nearrow&&\searrow&&\nearrow&&\searrow&&\nearrow&&\searrow&
&\nearrow&&\searrow&&\nearrow&&\searrow&&\nearrow&&\searrow&
&\nearrow&&\searrow&\cr
&\clap{A_1}&&&&\clap{\cdots}&&&&\clap{A_{n-1}}&&&&\clap{\Sigma}&&&&\clap{\Sigma}&&&&\clap{\Sigma}&&&&\clap{B_{n-1}}
&&&&\clap{\cdots}\cr
&&&\clap{\vdots}&&&&\clap{\vdots}&&&&\clap{\vdots}&&&&\clap{\vdots}&&&&\clap{\vdots}&&&&\clap{\vdots}&&&&\clap{\vdots}
\cr
}
$$
}
\end{enumerate}
\end{theorem}

\begin{proof}
(1) At two positions where the wing rims have the same non-zero terms, the modules have the same composition factors, by Proposition~\ref{homology-pattern}(2). Consider two adjacent such positions, such as the positions labelled $f,g$ in the diagram
{
\def\tempbaselines
{\baselineskip11pt\lineskip2pt
   \lineskiplimit2pt}
\def
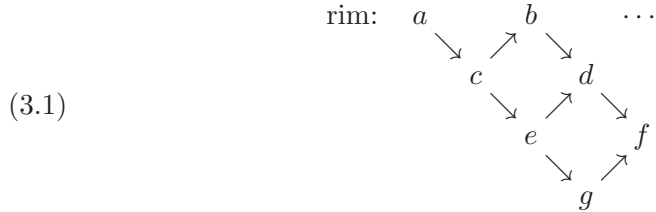
\diagram#1{\null\,\vcenter{\tempbaselines
\mathsurround=0pt
    \ialign{\hfil$##$\hfil&&\hskip5pt\hfil$##$\hfil\crcr
      \mathstrut\crcr\noalign{\kern-\baselineskip}
  #1\crcr\mathstrut\crcr\noalign{\kern-\baselineskip}}}\,}
\def\clap#1{\hbox to 0pt{\hss$#1$\hss}}
\begin{equation}\label{left-zeros}
\diagram{\hbox{rim:}\quad&\clap{a}&&&&\clap{b}&&&&\clap{\cdots}\cr
&&\searrow&&\nearrow&&\searrow&&\cr
&&&\clap{c}&&&&\clap{d}\cr
&&&&\searrow&&\nearrow&&\searrow\cr
&&&&&\clap{e}&&&&\clap{f}\cr
&&&&&&\searrow&&\nearrow\cr
&&&&&&&\clap{g}\cr
}
\end{equation}
}
In this example $g$ is southwest of $f$. There is another situation where $g$ is northwest of $f$, and it is handled similarly. Because $g$ and $f$ have the same non-zero terms on the wing rim, $a$ must be zero. Now the short exact sequence $0\to a\to c\to b\to 0$ shows that the map $c\to b$ is an isomorphism. In the short exact sequence $0\to c\to b\oplus e\to d\to 0$ the component $c\to b$ is an isomorphism, so the sequence splits and $e\to d$ is an isomorphism. Continuing, we deduce that $g\to f$ is an isomorphism, and we obtain a ladder of such isomorphisms. We have shown that the diagram has the form
{
\def\tempbaselines
{\baselineskip11pt\lineskip2pt
   \lineskiplimit2pt}
\def
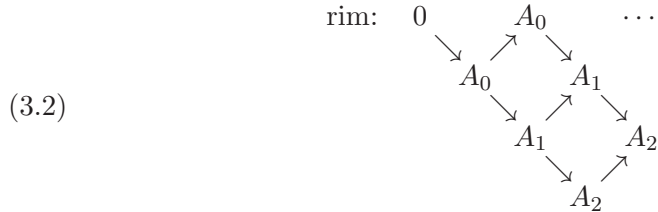
\diagram#1{\null\,\vcenter{\tempbaselines
\mathsurround=0pt
    \ialign{\hfil$##$\hfil&&\hskip5pt\hfil$##$\hfil\crcr
      \mathstrut\crcr\noalign{\kern-\baselineskip}
  #1\crcr\mathstrut\crcr\noalign{\kern-\baselineskip}}}\,}
\def\clap#1{\hbox to 0pt{\hss$#1$\hss}}
\begin{equation}\label{left-zeros}
\diagram{\hbox{rim:}\quad&\clap{0}&&&&\clap{A_0}&&&&\clap{\cdots}\cr
&&\searrow&&\nearrow&&\searrow&&\cr
&&&\clap{A_0}&&&&\clap{A_1}\cr
&&&&\searrow&&\nearrow&&\searrow\cr
&&&&&\clap{A_1}&&&&\clap{A_2}\cr
&&&&&&\searrow&&\nearrow\cr
&&&&&&&\clap{A_2}\cr
}
\end{equation}
}
for certain modules $A_0,A_1,A_2$ and where each of the mappings $A_i\to A_i$  is an isomorphism. 
This proves the second statement in (1), and the first statement holds because, between any two positions where the wing rims have the same non-zero modules, there is a path through adjacent such positions.

 %Proceeding in this way we establish the pattern shown on the left side of diagram~\ref{left-zeros}, and the pattern on the right adjacent to zeros on the rim of the homology diagram is established similarly.

(2) The entries on the rim of the homology diagram are the various homology modules of a single perfect complex, twisted by the Nakayama functor, so they are eventually zero to the left and to the right. At any given horizontal coordinate, sufficiently far from the rim the wing rim includes all non-zero homology modules, and so the homology modules at that distance are all isomorphic, by part (1). At this point all morphisms in the diagram are isomorphisms, and the homology has stabilized to a module $\Sigma$ whose composition factors are the union of the composition factors of all terms on the rim, by Proposition~\ref{homology-pattern}(2). 

(3) This follows from (1) and (2).
\end{proof}

\begin{definition}
For a perfect complex $P_\bullet$ over a self-injective algebra $\Lambda$, we will say that the $\Lambda$-module $\Sigma$ is the \textit{stabilization module} of $P_\bullet$ if, at each horizontal coordinate of the homology diagram of the Auslander-Reiten quiver component of $P_\bullet$, the modules stabilize with value $\Sigma$ sufficiently far from the rim.
\end{definition}

We now describe the homology diagram of a quiver component containing an indecomposable projective module $P_S$ and then use this to describe completely the entire quiver component. Recall that $H_S=\Rad P_S / \Soc P_S$ is the heart of $P_S$. We will see that $H_S$ is the stabilization module of $P_S$.

\begin{theorem}
\label{projective-homology-picture}
Let $S$ be a simple module for a self-injective algebra $\Lambda$ and let $P_S$  be the projective cover of $S$.  We assume $S\ne P_S$.
\begin{enumerate}
\item
The homology diagram of the quiver component containing $P_S$ is as follows:
{
\def\tempbaselines
{\baselineskip14pt\lineskip2pt
   \lineskiplimit2pt}
\def\diagram#1{\null\,\vcenter{\tempbaselines
\mathsurround=0pt
    \ialign{\hfil$##$\hfil&&\hskip10pt\hfil$##$\hfil\crcr
      \mathstrut\crcr\noalign{\kern-\baselineskip}
  #1\crcr\mathstrut\crcr\noalign{\kern-\baselineskip}}}\,}
\def\clap#1{\hbox to 0pt{\hss$#1$\hss}}
$$\diagram{&\clap{0}&&&&\clap{0}&&\clap{\times}&&\clap{P_S}&&\clap{\times}&&\clap{0}
&&&&\clap{0}&&\cr
&&\searrow&&\nearrow&&\searrow&&\nearrow&&\searrow&&\nearrow&
&\searrow&&\nearrow&\cr
\cdots&&&\clap{0}&&&&\clap{\Rad P_S}&&&&\clap{P_S/\Soc P_S}&&&&\clap{0}
&&&\cdots\cr
&&\nearrow&&\searrow&&\nearrow&&\searrow&&\nearrow&&\searrow&
&\nearrow&&\searrow&\cr
&\clap{0}&&&&\clap{\Rad P_S}&&&&\clap{H_S}&&&&\clap{P_S/\Soc P_S}
&&&&\clap{0}&&\cr
&&\searrow&&\nearrow&&\searrow&&\nearrow&&\searrow&&\nearrow&
&\searrow&&\nearrow&\cr
\cdots&&&\clap{\Rad P_S}&&&&\clap{H_S}&&&&\clap{H_S}&&&&\clap{P_S/\Soc P_S}&&&\cdots\cr
&&\nearrow&&\searrow&&\nearrow&&\searrow&&\nearrow&&\searrow&
&\nearrow&&\searrow&\cr
&\clap{\Rad P_S}&&&&\clap{H_S}&&&&\clap{H_S}&&&&\clap{H_S}&&&&\clap{P_S/\Soc P_S}\cr
&&&\clap{\vdots}&&&&\clap{\vdots}&&&&\clap{\vdots}&&&&\clap{\vdots}&&\cr
}
$$
}
Every mesh in the diagram labels the terms in a short exact sequence of zero homology modules, except for the two meshes labelled $\times$. In every case the short exact sequence is split except for the one underneath $P_S$, which is an Auslander-Reiten sequence.
\item The stabilization module of this quiver component is $H_S$.
\item
With the notation of Figure~\ref{AR-quiver-self-injective} the complex $C_n$ at distance $n$ from the rim of the quiver component containing $P_S$ is
$$
\nu^{-n} P_S\to\cdots\to\nu^{-2} P_S\to\nu^{-1} P_S\to P_S
$$
where each map sends the simple top of a projective module isomorphically to the socle of the next module. The irreducible morphisms between these complexes and their Auslander-Reiten translates are represented by the obvious inclusions and surjections between such complexes of adjacent lengths. 
\end{enumerate}
\end{theorem}

\begin{proof}
(1)
The fact that all meshes correspond to short exact sequences except the ones marked $\times$ follows from Lemma~\ref{split-long-exact-sequence-lemma} and the description of part of the quiver component given in Proposition~\ref{projectives-on-rim}, as in the proof of Proposition~\ref{homology-pattern}. The zero homology of the complexes shown in Proposition~\ref{projectives-on-rim} has already been computed before that proposition, and accounts for terms on the rim and the mesh below $P_S$. The facts that the remaining meshes correspond to split short exact sequence with maps between identically labelled terms being isomorphisms, and that the homology stabilizes at $H_S$, follow by the same argument as in Theorem~\ref{homology-picture}. We may deduce that the mesh below $P_S$ corresponds to an Auslander-Reiten sequence by observing from our earlier calculation that the component maps to $P_S$ and $H$ in the mesh are mono and epi, while the component maps in the mesh from $P_S$ and $H$ are epi and mono. This identifies the sequence as a well-known Auslander-Reiten sequence \cite[Sect. 4.11]{AR}. We will also see in Proposition~\ref{AR-complex} by a different argument that the short exact sequence under $P_S$ is an Auslander-Reiten sequence.

(2) follows from the picture in (1).

(3) We show that $C_n$ can be identified as stated by induction on $n$. We see from Proposition~\ref{projectives-on-rim} that the expressions for $C_0$, $C_1$ and $C_2$ are correct. Now suppose that $n\ge 2$, that the result is true for $n$ and smaller values, and consider $C_{n+1}$. From the general structure of the quiver shown in Figure~\ref{AR-quiver-self-injective} the mapping cone of $\nu^{-1}C_n\to C_n$ is $C_{n+1}\oplus\nu^{-1}C_{n-1}[1]$. The mapping cone has the form
$$
\nu^{-n-1} P_S\to(\nu^{-n} P_S)^2\to\cdots\to(\nu^{-2} P_S)^2\to(\nu^{-1} P_S)^2\to P_S
$$
and $\nu^{-1}C_{n-1}[1]$ has the form
$$
\nu^{-n} P_S\to\cdots\to\nu^{-2} P_S\to\nu^{-1} P_S
$$
so that 
$$
C_{n+1}\cong \nu^{-n-1} P_S\to\nu^{-n} P_S\to\cdots\to\nu^{-2} P_S\to\nu^{-1} P_S\to P_S.
$$
From the homology diagram in part (1) we see that the homology modules of $C_{n+1}$ are $\nu^{-n-1}\Rad P_S,\;\nu^{-n} H_S, \ldots,\nu^{-1} H_S, P_S/\Soc P_S$ and by a dimension count it follows that each map in $C_{n+1}$ sends the simple top of each projective to the socle of the next.
\end{proof}

\section{The position of short perfect complexes in the Auslander-Reiten quiver}

We deduce consequences of the patterns we have just seen. Throughout we assume that $\Lambda$ is a self-injective algebra with no semisimple summand and we consider the Auslander-Reiten quiver of perfect complexes.

\begin{corollary}\label{length-distance} 
If complexes on the rim of a quiver component have length $t$ then complexes at distance $r$ from the rim have length $t+r$.
\end{corollary}

\begin{proof}
This follows from Theorems \ref{homology-picture} and \ref{projective-homology-picture}, since the length of a complex is determined from its extreme non-vanishing homology modules, and so equals the length of the part of the row in the homology diagram where the terms are non-zero. These theorems show that this length grows by 1 with each step away from the rim.
\end{proof}

We now use this observation to restrict the distance from the rim of arbitrary perfect complexes: the distance of a perfect complex from the rim  of its quiver component is at most one less than its length. This bound is best possible, and we describe the complexes which achieve it. This extends the assertion in Proposition~\ref{projectives-on-rim} that complexes of length 1 lie on the rim.

\begin{corollary}\label{distance-n-complexes}
For each $n\ge 0$, the indecomposable perfect complexes of length $n+1$ all lie within distance $n$ from the rim of their quiver component.
The only perfect complexes of length $n+1$ which are at distance $n$ from the rim of their quiver component are the ones described in Theorem~\ref{projective-homology-picture} which appear in the component of a projective module $P_S$ associated to a simple module $S$. These are the complexes
$$
\nu^{-n} P_S\to\cdots\to\nu^{-2} P_S\to\nu^{-1} P_S\to P_S
$$
where every mapping sends a simple top isomorphically to the next simple socle. 
\end{corollary}

\begin{proof}
By Corollary~\ref{length-distance} all complexes of length $n+1$ must lie at distance $n$ or less from the rim. If such a complex lies at distance $n$ then by the same corollary the complexes on the rim of the quiver component have length 1, and so are projective modules $P_S$ in a single degree. We saw in Theorem~\ref{projective-homology-picture} that the complexes at distance $n$ from the rim in such a quiver component have the form stated.
\end{proof}

The following is a particular case of the last result.

\begin{corollary}\label{2-term-on-the-rim}
Any indecomposable perfect complex with two non-zero terms, not of the form $P_S\to \nu P_S$ for some simple module $S$ (where the map sends the top of $P_S$ isomorphically to the socle of $\nu P_S$), lies on the rim of its quiver component.
\end{corollary}

\begin{proof}
If it did not lie on the rim, the complexes on the rim of its quiver component would have to have length 1 by Corollary~\ref{length-distance}, so would be shifts of a projective module $P_S$ for some simple module $S$. In that case the 2-term complexes next to the rim have the form which has been excluded, by Proposition~\ref{projectives-on-rim}.
\end{proof}

\begin{corollary}\label{homology-string-length}
In an indecomposable perfect complex at distance $n$ from the rim, any non-zero homology module must occur as part of a string of at least $n+1$ non-zero consecutive homology modules. 
\end{corollary}

\begin{proof}
The result is true for quiver components containing a projective module by Theorem~\ref{projective-homology-picture}. For other components,
non-zero homology comes about from non-zero homology on the rim, by Proposition~\ref{homology-pattern}. Each non-zero term on the rim gives rise to $n+1$ non-zero consecutive homology modules at distance $n$. This is the only way we get non-zero homology at distance $n$ from the rim and so every non-zero homology module must be part of a string of $n+1$ non-zero homology modules.
\end{proof}

We show that the condition of Corollary~\ref{homology-string-length} forces certain complexes to lie on, or close to, the rim.

\begin{corollary}\label{0H0-criterion}
Any indecomposable perfect complex $X$ with a sequence of homology modules of the form $H_{d+1}(X)=0$, $H_{d}(X)\ne 0$ and $H_{d-1}(X)=0$ must lie on the rim of its quiver component.
\end{corollary}

\begin{proof}
If $X$ were not on the rim, non-zero homology strings would all have to have length at least 2 by Corollary~\ref{homology-string-length}. Since the complex has a non-zero homology string of length 1, it must lie on the rim.
\end{proof}

\begin{proposition}
\label{truncated-resolutions}
Let $M$ be an indecomposable module, and let $P_n\to \cdots\to P_1\to P_0\to M\to 0$ be the start of a minimal resolution of $M$. Then the complex $P_n\to \cdots\to P_1\to P_0$ is indecomposable. It lies on the rim of its quiver component unless $M=\Omega^{-1}S$ for some simple module $S$ and $n=1$, in which case the complex lies at distance 1 from the rim of its quiver component.
\end{proposition}

\begin{proof} If the complex were to decompose as a direct sum of two complexes, one of the summands would have $M$ as its zero homology. That summand would have projective terms, and would be the start of a smaller projective resolution of $M$. Since we chose a minimal projective resolution, this cannot happen.

If $n>1$ then the homology of the complex has its zero homology group isolated in the sense of Corollary~\ref{0H0-criterion} and so the complex lies on the rim of its quiver component. The other possibility is that $n\le 1$ so that the complex has one or two terms. If it has one term ($n=0$) then $M$ is projective and lies on the rim of the quiver by Proposition~\ref{projectives-on-rim}. If $n=1$ then by hypothesis the complex is not isomorphic to $P_S\to \nu P_S$ and so lies on the rim by 
Corollary~\ref{2-term-on-the-rim}.
\end{proof}

\section{An application: complexes with large homology}

The following result is of independent significance, in that it makes no reference to the Auslander-Reiten quiver.

\begin{theorem}
\label{big-homology-complexes}
Let $\Lambda$ be a self-injective algebra of radical length at least 3. There exist indecomposable perfect complexes with degree zero homology of arbitrarily large dimension.
\end{theorem}

\begin{proof}
We may construct a perfect complex with arbitrarily many non-zero homology modules, and  for some $H\ne0$ having terms $0,H,0$ somewhere in the list of homology groups. For example, we may take an indecomposable projective module $P_S$ of radical length at least 3, form a complex 
$$P_{\Rad P_S}\to P_S\to\nu P_S\to \nu^2 P_S\to\cdots\to $$
where the first map is the projective cover of the radical of $P_S$, and after that each map identifies the top of a module with the socle of the next. Such a complex
can be made to have arbitrarily many non-zero homology modules, and the second homology module from the left is zero.
Such a complex must lie on the rim of its quiver component by Corollary~\ref{0H0-criterion}. Considering homology at some distance from the rim the stabilizing module $\Sigma$ which appears in Corollary~\ref{homology-picture} has composition length which is the sum of the lengths of all homology modules of the starting complex. We conclude that this may be made arbitrarily large.
\end{proof}

It is perhaps surprising that even algebras as small as $k[t]/(t^3)$ (and $k[t]/(t^n)$ with $n\ge 3$)  have large indecomposable perfect complexes with $H_0$ of arbitrarily large dimension, and it conveys the sense that there are very many such complexes for these algebras. By contrast, the indecomposable perfect complexes for $k[t]/(t^2)$ are well known to form a single quiver component: the component containing the unique projective, described explicitly in Theorem~\ref{projective-homology-picture}.

The result has consequences beyond the context of self-injective algebras and, as an example, we present an application to the theory of Mackey functors. For background on this theory we refer to \cite{TW} and \cite{Webc}, and we review what we need here but we will be brief on the details. Our example has a connection with equivariant homotopy theory, for which we refer to \cite{DHM}. We consider cohomological Mackey functors for the cyclic group $G=C_p$  over a field $k$ of characteristic $p$. The structure of these Mackey functors is considered explicitly in \cite[Sec. 4]{BSW}. In \cite{DHM} the perfect complexes of such functors are computed when $p=2$, noting that they refer to cohomological Mackey functors  as $\underline \ZZ$-modules. 

There are two simple Mackey functors for $G$ over $k$, denoted $S_{1,k}$ and $S_{G,k}$ in \cite{TW}, with evaluations  $S_{1,k}(1)=k$, $S_{1,k}(G)=0$ and $S_{G,k}(1)=0$, $S_{G,k}(G)=k$, and they are cohomological. Let us write $a= S_{1,k}$ and $b=S_{G,k}$ for short. Their projective covers have the structure:
$$P_a=FP_k=
\mathop{
\begin{xy}
(0,2)*{b};
(0,-2)*{a};
\end{xy}
}
\quad\hbox{and}\quad
P_b=FP_{kG}= \mathop{
\begin{xy}
(0,8)*{a};
(-4,4)*{a};
(-4,1)*{\vdots};
(4,0)*{b};
(-4,-4)*{a};
(0,-8)*{a};
\end{xy}
}
$$
where $P_b$ has $p$ composition factors of type $a$.

\begin{theorem}
\label{cohomological-MF-theorem}
Let $k$ be a field of characteristic $p$ and $G=C_p$ a cyclic group of order $p$. If $p\ge 3$ there are perfect complexes of cohomological Mackey functors with zero homology of arbitrarily large dimension.
\end{theorem}

\begin{proof}
Cohomological Mackey functors for $G$ are the same as modules for the \textit{cohomological Mackey algebra} for $G$, defined in \cite{TW}, and which we denote $\Lambda$ here. We have $\Lambda = P_a\oplus P_b$ as left $\Lambda$-modules, and this decomposition corresponds to an expression $1 = e_a + e_b$ as a sum of orthogonal idempotents, so that $P_a\cong \Lambda e_a$ and $P_b\cong \Lambda e_b$. Thus $e_a\Lambda e_a \cong \End_\Lambda(P_a)^\op \cong k[t]/(t^p)$. If $M$ is a cohomological Mackey functor on $G$ the specification $M\mapsto e_aM$ provides an exact functor from $\Lambda$-modules to $k[t]/(t^p)$-modules.

Given a perfect complex $P_\bullet$ of $k[t]/(t^p)$-modules we can lift it to a perfect complex $\hat P_\bullet$ of cohomological Mackey functors, that maps to $P_\bullet$ under this functor, by replacing each copy of $k[t]/(t^p)$ by $P_a$, and by lifting uniquely the homomorphisms between copies of $k[t]/(t^p)$ to homomorphisms between the $P_a$. We may check that this does give a chain complex. Multiplying the homology modules of $\hat P_\bullet$ by $e_a$ we obtain the homology modules of $P_\bullet$. Because we can construct these to have arbitrarily large dimension, the homology modules of the $\hat P_\bullet$ will also have arbitrarily large dimension.
 \end{proof}

\section{Short complexes and Auslander-Reiten sequences}
\label{length-3-complexes}

In this section we study in detail complexes of lengths 2 and 3, showing that they are closely related to Auslander-Reiten sequences. We find their position in the Auslander-Reiten quiver, and show the stabilization modules of complexes of length 2 are the middle terms in Auslander-Reiten sequences.

We start with a result that makes a connection between Auslander-Reiten sequences and Auslander-Reiten triangles. It is closely related to a result in \cite{DPW2}. For this result $\Lambda$ can be any finite dimensional algebra, not necessarily one that is self-injective. We write $\tau$ for the Auslander-Reiten translate.

\begin{proposition}\label{AR-complex}
Let $\Lambda$ be a finite dimensional algebra over $k$ and $M$ a non-projective indecomposable $\Lambda$-module. Let $\calP_M = (P_1\to P_0)$ be the first two terms of a minimal projective resolution of $M$. Consider the Auslander-Reiten triangle $\nu\calP_M[-1]\to \calE_M\to\calP_M\to\nu\calP_M$ in $D^b(\Lambda\hbox{-mod})$.  
\begin{enumerate}
\item The part of the long exact sequence in homology of this triangle, whose terms are $H_0$, is the Auslander-Reiten sequence of $\Lambda$-modules $0\to\tau M\to E_M\to M\to 0$. 
\item The non-zero homology groups of the mapping cone complex $\calE_M$ are $H_1(\calE_M)=\Omega^2 M$, $H_0(\calE_M)=E_M$ and $H_{-1}(\calE_M)= \nu M$.
\end{enumerate}
\end{proposition}

\begin{proof}
(1) We know from \cite[IV.2.4]{ASS} that $\nu \calP_M[-1]$ has homology modules 
$$H_{-1}(\nu\calP_M[-1])=\nu M
\qquad\hbox{and}\qquad 
H_0(\nu\calP_M[-1])=\tau M.
$$
By Lemma~\ref{split-long-exact-sequence-lemma}, because $M$ is not projective, the long exact sequence in homology of the triangle breaks up as a splice of short exact sequences, of which the 0-homology sequence has the form
$$
0\to \tau M\to H_0(\calE_M)\to M\to 0.
$$
This sequence has the correct end terms to be an Auslander-Reiten sequence and, in particular, they are indecomposable. We show first that the sequence has the lifting property of Auslander-Reiten sequences, and then later that it is not split. This will establish that it is indeed an Auslander-Reiten sequence.

Suppose that we have an indecomposable module $N$ and a morphism $N\to M$ which is not an isomorphism. Take the start of a minimal projective resolution $\calQ = (Q_1\to Q_0)$ of $N$ and lift the morphism $N\to M$ to a morphism of complexes $\calQ\to \calP_M$. This mapping is not split epi, because on taking zero homology it is not split epi. It therefore lifts to a morphism $\calQ\to\calE_M$ and on taking zero homology we deduce that the original morphism lifts to a morphism $N\to H_0(\calE_M)$.

To show that the sequence of zero homology groups is not split, suppose to the contrary that we have a splitting map $M\to H_0(\calE_M)= Z_0(\calE_M)/B_0(\calE_M)$ with image $U/B_0(\calE_M)$ for some submodule $U\subseteq Z_0(\calE_M)$. By the projective property of $P_0$ this lifts to a map $P_0\to U$ and hence also to a map $P_1\to (\calE_M)_1$, so that we have a chain map $\calP_M\to\calE_M$. Composing with the map $\calE_M\to\calP_M$ we obtain an endomorphism of $\calP_M$ which lifts the identity on $M$. By Fitting's Lemma some power of this endomorphism has image a summand of $\calP_M$, which again must be the start of a resolution of $M$. By minimality of $\calP_M$ this summand must be the whole of $\calP_M$, and so the endomorphism of $\calP_M$ is an isomorphism and we have split the Auslander-Reiten triangle -- a contradiction.

(2) We have just shown that taking zero homology gives an Auslander-Reiten sequence, so this gives the value of $H_0(\calE_M)$. To identify the top and bottom homology groups of $\calE_M$ we use the fact again from Lemma~\ref{split-long-exact-sequence-lemma} that the degree 1 and degree $-1$ terms in the long exact homology sequence give short exact sequences, and they are
$$
0\to 0\to H_1(\calE_M)\to H_1(\calP_M)=\Omega^2 M\to 0
$$
and
$$
 0\to H_{0}(\nu\calP_M)=\nu M \to H_{-1}(\calE_M)\to 0\to 0.
$$
This completes the proof.
\end{proof}

The construction just described has potential as a way to organize the automatic computation of Auslander-Reiten sequences,  without explicitly computing Ext groups or extensions corresponding to Ext classes. Given a module $M$, compute the first two terms $\calP_M=(P_1\to P_0)$ in a minimal projective resolution of $M$, and also the effect of the Nakayama functor $\nu\calP_M$. The next part is hard: compute a chain map $\calP_M\to\nu\calP_M$ that is not homotopic to zero, and which is annihilated up to homotopy by $\Rad(\End(\calP_M))$. Finally compute the mapping cone $\calE_M[1]$ of this map and compute the zero homology sequence of the triangle $\nu\calP_M[-1]\to\calE_M\to\calP_M\to \nu\calP_M$. This computes the Auslander-Reiten sequence ending at $M$.

The last result allows us to identify the stabilization module of any indecomposable perfect complex of length 2 over a self-injective algebra. 

\begin{corollary}
\label{length-2-stabilization-module}
Let $\Lambda$ be a self-injective algebra and $M$ a non-projective indecomposable $\Lambda$-module that is not of the form $\Omega^{-1}(S)$, for some simple module $S$. Let $\calP_M = (P_1\to P_0)$ be the first two terms of a minimal projective resolution of $M$. Then $\calP_M$ lies on the rim of its quiver component, and the homology diagram of this component has the form
{
\def\tempbaselines
{\baselineskip16pt\lineskip3pt
   \lineskiplimit3pt}
\def
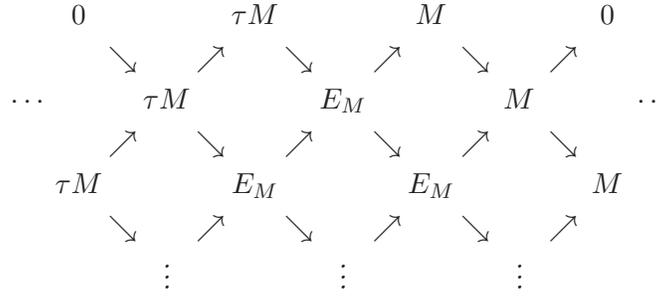
\diagram#1{\null\,\vcenter{\tempbaselines
\mathsurround=0pt
    \ialign{\hfil$##$\hfil&&\quad\hfil$##$\hfil\crcr
      \mathstrut\crcr\noalign{\kern-\baselineskip}
  #1\crcr\mathstrut\crcr\noalign{\kern-\baselineskip}}}\,}

\def\clap#1{\hbox to 0pt{\hss$#1$\hss}}
$$
\diagram{&\clap{0}&&&&\clap{\tau M}&&&&\clap{M}&&&&\clap{0}\cr
&&\searrow&&\nearrow&&\searrow&&\nearrow&&\searrow&&\nearrow&\cr
\cdots&&&\clap{\tau M}&&&&\clap{E_M}&&&&\clap{M}&&&\cdots\cr
&&\nearrow&&\searrow&&\nearrow&&\searrow&&\nearrow&&\searrow&\cr
&\clap{\tau M}&&&&\clap{E_M}&&&&\clap{E_M}&&&&\clap{M}\cr
&&\searrow&&\nearrow&&\searrow&&\nearrow&&\searrow&&\nearrow&\cr
&&&\clap{\vdots}&&&&\clap{\vdots}&&&&\clap{\vdots}\cr
}
$$
}
where $0\to\tau M\to E_M\to M\to 0$ is the Auslander-Reiten sequence terminating at $M$. Thus the stabilization module of $\calP_M$ is $E_M$.
\end{corollary}

Note that we already identified the stabilization module of $\calP_{\Omega^{-1}(S)}$ as $H_S$ in Theorem~\ref{projective-homology-picture}, and this is the middle term of the Auslander-Reiten sequence terminating at $\Omega^{-1}(S)$, with the projective summand removed.

\begin{proof}
We know that $\calP_M$ lies on the rim of the quiver by Corollary~\ref{distance-n-complexes}, so that the rim is as described, and the second row is a consequence of Proposition~\ref{AR-complex}. We fill in the rest of the diagram using Theorem~\ref{homology-picture}.
\end{proof}

When $\Lambda$ is a symmetric algebra, the modules in an Auslander-Reiten sequence appear not only as the zero homology of complexes in a triangle, but also as the homology groups of a single complex.

\begin{corollary}
\label{AR-as-homology}
If $\Lambda$ is a symmetric algebra and $M$ is a non-projective indecomposable $\Lambda$-module, there is a 3-term perfect complex $\calE_M$ with $H_1(\calE_M)=\tau M$, $H_0(\calE_M)=E$ and $H_{-1}(\calE_M)= M$, where $0\to\tau M\to E_M\to M\to 0$ is the Auslander-Reiten sequence of $\Lambda$-modules terminating at $M$
\end{corollary}

\begin{proof}
When $\Lambda$ is symmetric the Nakayama functor is the identity and $\tau M\cong\Omega^2 M$. The complex $\calE_M$ of Proposition~\ref{AR-complex} has the desired properties.
\end{proof}

%It is not always possible to construct perfect complexes with prescribed homology groups. One constraint on the homology groups is that their alternating sum must lie in the span of projective modules in the appropriate Grothendieck group. Even when this condition is satisfied, it is still not always possible to realize a list of groups as homology groups, and we mention in this context the result of Carlsson and Allday-Puppe, for which a different approach has been given by Avramov, Buchweitz, Iyengar and Miller~\cite{ABIM}: if $\calP$ is a perfect complex of $kG$-modules for an elementary abelian $p$-group $G=C_p^r$, where $k$ is a field of characteristic $p$, then
%\babble{State as a corollary that the sum of lengths of terms in an AR sequence is $\ge r+1$. In fact $M$ and $\Omega^2M$ are the homology of a complex, so it is weaker to say that the terms of an ARS are the homology of a complex.}
%$$\sum_i \hbox{Loewy length }H_i(\calP)\ge r+1.$$
%This non-existence result complements the content of Corollary~\ref{AR-as-homology} and allows us to see that even with three modules which appear in a short exact sequence, it is not always possible to realize those modules as the homology of a perfect complex of length 3.

We can now characterize the positions of complexes of length 3 in the Auslander-Reiten quiver. Given a non-projective indecomposable module $M$, we continue with the notation that $\calE_M$ is the complex that appears in the Auslander-Reiten triangle $\nu\calP_M[-1]\to\calE_M\to\calP_M\to \nu\calP_M$.

\begin{proposition}
\label{3-term-complexes-position}
Let $\Lambda$ be a self-injective algebra and $M$ an indecomposable non-projective $\Lambda$-module. 
\begin{enumerate}
\item
The complex $\calE_M$ is indecomposable unless $M=\Omega^{-1}S$ for some simple module $S$, in which case $\calE_M$ is the direct sum of an indecomposable 3-term complex and $P_S$ as a complex in degree 0. 
\item
An indecomposable 3-term perfect complex lies on the rim of its quiver component unless, up to shift, it is either of the form $\calE_M$ for some $M\not\cong\Omega^{-1}S$ (in which case it lies at distance 1 from the rim) or it is the indecomposable 3-term summand of $\calE_{\Omega^{-1}S}$ for some simple $\Lambda$-module $S$ (in which case it lies at distance 2 from the rim).
\end{enumerate}
\end{proposition}

\begin{proof}
(1) The complexes $\calP_M$ lie on the rim of the quiver by Corollary~\ref{2-term-on-the-rim} unless $M=\Omega^{-1}S$ for some simple module $S$. Apart from this exception, $\calE_M$ is the middle term of the Auslander-Reiten triangle with third term $\calP_M$, and it is indecomposable since $\calP_M$ lies on the rim. When $M=\Omega^{-1}S$ we have seen the decomposition of $\calE_M$ in Theorem~\ref{projective-homology-picture}.

(2) We have seen from Corollary~\ref{distance-n-complexes} that indecomposable 3-term complexes lie at distance at most 2 from the rim, and the ones at distance 2 have the form $\nu^{-1}P_S\to P_S\to \nu P_S$. If a 3-term complex lies at distance 1 from the rim, the complexes on the rim must have length 2, and are (up to shift) of the form $\calP_M$ with $M\not\cong\Omega^{-1}S$ for any simple $S$, by Corollary~\ref{2-term-on-the-rim}. Apart from these cases, 3-term complexes must lie on the rim.
\end{proof}

Thus, if a 3-term perfect complex has homology modules which are not the terms in an Auslander-Reiten sequence of $\Lambda$-modules (to within a projective summand of the middle term) then the complex lies on the rim of the Auslander-Reiten quiver of perfect complexes.

\section {The position of rigid complexes in the Auslander-Reiten quiver}
\label{rigid-complexes}

It seems more often than not when we consider a class of objects satisfying some given properties, that those objects lie on the rim of the Auslander-Reiten quiver, or close to the rim. 
We have seen this with short complexes, with complexes with zero homology in certain degrees, and in \cite[Theorem 6.2]{Webb} an example is given with complexes with small endomorphism rings. 
We give another of these examples, showing that over a symmetric algebra an indecomposable perfect complex $C$ which is \textit{rigid}, i.e. $\Hom(C,C[1])=0$, must lie on the rim of the quiver. Thus, as a particular case, summands of tilting complexes lie on the rim. This particular case can also be deduced from Rickard's theory \cite{Ric}, given that projective modules lie in the rim (Proposition~\ref{projectives-on-rim}). 

We will deduce our result from the following lemma on Auslander-Reiten triangles, whose analogue for Auslander-Reiten sequences is also true by the same argument. 

\begin{lemma}
\label{zero-composite-lemma}
Let $A\to B\to C\to A[1]$ be an Auslander-Reiten triangle in a Krull-Schmidt triangulated category $\calC$ and suppose that $\Hom(A,C)=0$ Then $B$ is indecomposable, so that $A$ and $C$ lie on the rim of the Auslander-Reiten quiver of $\calC$.
\end{lemma}

\begin{proof}
Suppose to the contrary that $B=B_1\oplus B_2$. We show that it is not possible for both of the composite maps $A\to B_1\to C$ and $A\to B_2\to C$ to be zero as follows. Let $W$ be any object and consider the long exact sequence
$$
\xymatrix@C=10pt{
&&
\hskip3.7em\clap{\dots}\hskip3.7em\ar[r]&\Hom(W,C[-1])
\ar `r[d] `[l] `[llld] `[dll] [dll]\\
&\Hom(W,A) \ar[r]^(.33)\alpha& \Hom(W,B_1)\oplus\Hom(W,B_2) \ar[r] ^(.63){\beta}& \Hom(W,C)
\ar `r[d] `[l] `[llld] `[dll] [dll]\\
&\Hom(W,A[1]) \ar[r] &
\hskip3.7em\clap{\cdots}\hskip3.7em & }
$$
If both composites $A\to B_i\to C$ were zero for $i=1,2$, then both composites $\Hom(W,C)\to\Hom(W,B_i)\to \Hom(W,C)$ would be zero. It would follow that both projections of $\ker\beta=\Im\alpha$ onto the two summands $\Hom(W,B_i)$ lie in $\ker\beta$ and hence
$$
\alpha\Hom(W,A)=(\ker\beta\cap\Hom(W,B_1))\oplus (\ker\beta\cap\Hom(W,B_2))
$$
is a direct sum of $\End(W)$-modules. Take $W=A$. Because $\End(A)$ is local we deduce that $\alpha\Hom(A,A)$ has a unique simple quotient as an $\End A$-module. From this it follows that one of the summands $\ker\beta\cap\Hom(A,B_i)$ is zero and that $\alpha\Hom(A,A)\subseteq\Hom (A,B_j)$ for the other suffix $j$. Now $\alpha(1_A)$ is represented by the two component maps $A\to B_1$ and $A\to B_2$, so one of these must be zero. However, these component maps are irreducible morphisms and are never zero. This contradiction shows that one of the composite maps $A\to B_1\to C$ and $A\to B_2\to C$ is non-zero (and hence both are). We complete the proof by invoking the hypothesis that there is no non-zero map $A\to C$. This further contradiction establishes that $B$ is indecomposable.
\end{proof}

We can now deduce our application to rigid complexes.

\begin{theorem}
\label{rigid-complex-theorem}
Let $C$ be an indecomposable perfect complex for a symmetric algebra $\Lambda$ with the property that $\Hom(C,C[1])=0$ Then $C$ lies on the rim of its component in the Auslander-Reiten quiver.
\end{theorem} 

\begin{proof}
The Auslander-Reiten triangle starting at $C$ has the form $C\to X\to C[1]\to C[1]$ since the Nakayama functor is the identity. Because $\Hom(C,C[1])=0$ the result follows from Lemma~\ref{zero-composite-lemma}.
\end{proof}

\end{document}